\documentclass[12pt]{article}
\usepackage{amssymb,amsmath,amsthm,secdot,enumitem,bbm}
\sectiondot{subsection}
\usepackage{xcolor}
\usepackage{authblk}
\DeclareMathOperator{\I}{\mathbbm{1}}%

\DeclareMathOperator{\Law}{Law}%
\DeclareMathOperator{\sign}{sign}

\newcommand{\norm}[1] {|\!|\!|#1|\!|\!|}

\def\E{\hskip.15ex\mathsf{E}\hskip.10ex}
\def\P{\mathsf{P}}

\def\eps{\varepsilon}
\def\phi{\varphi}

\newtheorem{Theorem} {Theorem}[section]
\newtheorem{Lemma}[Theorem]{Lemma}
\newtheorem{Proposition}[Theorem]{Proposition}

\newtheorem{Assumption}{Assumption}

\theoremstyle{definition}\newtheorem{Example}{Example}[section]
\theoremstyle{definition}\newtheorem{Remark}[Theorem]{Remark}
\theoremstyle{definition}

\numberwithin{equation}{section}

\renewcommand{\ge}{\geqslant}
\renewcommand{\le}{\leqslant}

\newcommand{\nn}{\nonumber}
\newcommand{\wt}{\widetilde}

\renewcommand{\d}{\partial}
\newcommand{\B}{\mathcal{B}}
\newcommand{\C}{\mathcal{C}}
\newcommand{\F}{\mathcal{F}}
\newcommand{\G}{\mathcal{G}}
\newcommand{\N}{\mathbb{N}}
\newcommand{\R}{\mathbb{R}}
\newcommand{\Z}{\mathbb{Z}}

\title{\textbf{Invariant measures for stochastic functional differential equations}}

\author[$*$,1,2]{Oleg Butkovsky}
\author[$\ddagger$,2]{Michael Scheutzow}
\affil[1]{\small{Technion --- Israel Institute of Technology,

Faculty of Industrial Engineering and Management

Haifa, 3200003, Israel.\bigskip
}}

\affil[2]{\small {Technische Universit\"at Berlin,

Institut f\"ur Mathematik, MA 7-5, Fakult\"at II,

Stra\ss e des 17.~Juni 136, 10623 Berlin, FRG}}

\date{October 30, 2017} 

\begin{document}
\maketitle
\renewcommand{\thefootnote}{*}
\footnotetext{Supported in part by Israel Science Foundation Grant 1325/14,
DFG Research Unit FOR 2402, National Science Foundation Grant DMS--1440140,
and a Technion fellowship.\\ Email: \texttt{oleg.butkovskiy@gmail.com}.}
\renewcommand{\thefootnote}{$\ddagger$}
\footnotetext{Supported in part by DFG Research Unit FOR 2402,
National Science Foundation Grant DMS--1440140. Email: \texttt{ms@math.tu-berlin.de}.}

\begin{abstract}
We establish new general sufficient conditions for the existence of an invariant measure for
stochastic functional differential equations and exponential or subexponential convergence to the equilibrium.
The obtained conditions extend the Ve\-re\-tennikov--Khasminskii conditions for SDEs and are optimal in a certain sense.
\end{abstract}

\section{Introduction}

While ergodic properties of stochastic differential equations (SDEs) are more or less understood by now, less is known about  ergodic properties of stochastic functional (or~delay) differential equations (SFDEs). In this article we establish new general sufficient conditions for existence of an invariant measure for SFDEs and obtain estimates for the rate of convergence to the equilibrium.

SFDEs in general have quite a peculiar ergodic behavior that can be very different from the ergodic behavior of SDEs. Let us briefly describe the main features. First of all, as was shown in \cite{Sch05}, an SFDE might have a \textit{reconstruction property}. Namely, consider the equation
\begin{align}\label{badsfde}
d X^{(x)}(t)&= f(X^{(x)}(t-1))dt + g(X^{(x)}(t-1))d W(t),\quad t\ge0,\\
X^{(x)}(t)&=x(t),\quad t\in[-1,0]\nn.
\end{align}
where $f\colon\R\to\R$ is a Lipschitz function, $g\colon\R\to\R$ is a positive strictly increasing bounded Lipschitz function, $x\colon[-1,0]\to\R$ is a continuous function and $W$ is a 1-dimensional Brownian motion. It turns out that if for some $N>0$ (that might be arbitrarily large) one observes a single piece of trajectory $\{X(t,\omega),\, t\in[N,N+1]\}$, then with probability $1$ one can reconstruct the initial condition $\{X(t),\, t\in[-1,0]\}$ of the SFDE. Clearly, this is not the case for SDEs.

As a result of this, the solution to \eqref{badsfde} is not a strong Feller process and it does not have a \textit{mixing property}. Indeed, if $x\neq y$, then the measures $\Law \{X^{(x)}(t),\, t\in[N,N+1]\}$ and $\Law \{X^{(y)}(t),\, t\in[N,N+1]\}$ are mutually singular for any $N>0$. Therefore, one cannot hope to construct a classical coupling between these measures to show asymptotic stability.

SFDEs might also have a \textit{resonance property}. If one considers a delay version of a classical Ornstein--Uhlenbeck process,
\begin{equation}\label{OU}
d X(t)=-\lambda X(t-1)\,d t + d W(t),\quad t\ge0,
\end{equation}
where $\lambda>0$, then (contrary to the non-delay case) for large enough $\lambda$ (more precisely, $\lambda\ge\pi/2$) this equation does not have an invariant measure \cite{KM92}. Moreover, for large $\lambda$ the equation oscillates to infinity with rapidly increasing diameter of oscillations.

Due to the above mentioned challenges the question of existence of an invariant measure and rate of convergence to the equilibrium remained open even for a relatively simple SFDE \eqref{badsfde} if $f$ is not affine. In the current paper we present an answer to this question.

Let us recall that there are two quite general approaches that are used to study the ergodic properties of Markov processes. The first approach is based on  functional inequalities, see, e.g., \cite{BCG}. The second approach
is based on the concept of small sets and utilizes the coupling method, see, e.g., \cite{MT09}. Using these techniques, it was shown that if the drift vector field of an SDE points towards the origin (the so-called
Veretennikov--Khasminskii condition), then, under some further non-degeneracy assumptions, the SDE has a unique invariant measure and converges to it in total variation, see \cite{DFG}, \cite{KV05}, \cite{Ver00}. More general SDEs are treated in~\cite{HM15}.

Unfortunately, these methods are not applicable for SFDEs due to their lack of mixing properties. Note
though that  ergodic properties of affine SFDEs can be treated by comparison with the deterministic case and by studying the fundamental solutions, see \cite{GK00}, \cite{KM92}, \cite{MS90}, \cite{RRG}.  However this technique also
does not work for non-affine SFDEs.

Some sufficient conditions for the existence of an invariant measure for SFDEs are obtained in \cite[Theorem~3]{IN}. Let us note though that it might be quite hard to verify these conditions in practice.

To overcome these difficulties and to derive verifiable sufficient conditions M.~Hairer, J.~Mattingly and M.~Scheutzow suggested a new approach targeted specifically at Markov processes with bad mixing properties \cite{HMS11}. They introduced a new concept of a
$d$-small set, and showed that under certain conditions (much weaker than mixing) a Markov process has a unique invariant measure and converges to it. The price to pay is that this convergence occurs in the Wasserstein metric rather than in total variation. This approach was further developed in \cite{Bu14}.

In this paper we apply this general approach to SFDEs. The main obstacle here is to construct a proper Lyapunov function. Due to the memory property it is much more challenging than in the SDE case. Indeed, a solution to an SFDE is an infinite dimensional Markov process with non-locally compact state space and rather involved generator. We develop a new technique inspired by some ideas from \cite{Sch84}.

Another obstacle was to obtain a condition that is general enough to cover drifts in \eqref{badsfde} of the form $f(x)=-|x(-1)|^{\beta}\sign (x(-1))$, $\beta\in[0,1)$ (in this case an invariant measure exists), but not ``too general'' since \eqref{badsfde} with the drift $f(x)=-\lambda x(-1)$ does not have an invariant measure for  $\lambda\ge\pi/2$.

The obtained result can be formulated as follows: one should check that the drift vector field $f(x)$ points towards the origin only for ``typical''  $x$. This extends and generalizes the corresponding theorems for SFDEs in
\cite{Sch84}, \cite{HMS11}, \cite{Bu14}.  The obtained conditions and rates are optimal in a certain sense. We explain our result in more details below in Section~2.

Note also that there is an alternative fruitful approach, which is also suitable for SFDEs, that was suggested and developed in \cite{HMS11}, \cite{KS15}. It is based on the generalized coupling method. Using this approach it is possible to establish uniqueness of an invariant measure and asymptotic stability under some natural conditions. However, this approach does not allow directly to obtain the results on existence of an invariant measure and on the convergence rate. Therefore we do not use it here.

The paper is organized as follows. We formulate and discuss our main results in Section~2. Section~3 contains specific applications of our results to different SFDEs as well as some counterexamples. All proofs are placed in
Section~4.

\medskip
\noindent \textbf{Acknowledgments}. The authors are very grateful to Alexei Kulik for fruitful discussions. The authors also would like to thank the referees for their valuable comments and suggestions which helped to improve the quality of the paper.

This material is partially based upon work supported by the National Science Foundation under Grant No. DMS--1440140 while the authors were in residence at the Mathematical Sciences Research Institute (MSRI) in Berkeley, California, during the Fall 2015 semester. The authors continued the project during the visit of OB to Technische Universit\"at Berlin in March 2016. The authors are  grateful to MSRI and TU Berlin for their support and hospitality.

\section{Main results}\label{S:2}

We assume that all random objects are defined on a common probability space $(\Omega,\F,\P)$. Fix $r>0$, positive integers $d$, $m$ and let $\C:=\C([-r,0],\R^d)$ be the space of continuous functions endowed with the supremum norm $\|\cdot\|$. We study a stochastic functional differential equation
\begin{align}\label{mainSDDE}
d X(t)&=f(X_t)\,d t + g(X_t)\,d W(t),\quad t\ge0\\
X_0&=x\nn
\end{align}
where  $f:\C \to \R^d$ and $g:\C \to \R^{d\times m}$ are measurable functions, $W$ is an $m$-dimensional Brownian motion, the initial condition $x\in\C$, and we used the standard notation $X_t(s):=X(t+s)$, $s\in[-r,0]$.

For a matrix $M\in\R^{d\times m}$ we denote by $\norm{M}$ its Frobenius norm, that is, $\norm{M}:=\sqrt{\sum M_{ij}^2}$. For a real $a$ we put $a_+:=\max(a,0)$. We suppose that the drift and diffusion of \eqref{mainSDDE} satisfy the following condition:
\begin{Assumption}\label{A:1}
The drift $f$ is continuous and bounded on bounded 
subsets of $\C$. The diffusion $g$ is non-degenerate, that is, for any $x\in\C$ the matrix $g(x)$ admits a right inverse $g^{-1}(x)$ and
\begin{equation*}
\sup_{x\in\C}\norm{g^{-1}(x)}<\infty.
\end{equation*}
Furthermore, $f$ satisfies the one-sided Lipschitz condition and $g$ is Lipschitz. Namely, there exists $C>0$ such that for any $x,y\in\C$ we have
\begin{equation*}
\langle f(x)-f(y),x(0)-y(0)\rangle_+ +\norm{ g(x)-g(y)}^2\le C \|x-y\|^2.
\end{equation*}
\end{Assumption}

It follows from \cite{RS08}  that under Assumption~\textbf{\ref{A:1}} SFDE \eqref{mainSDDE} has a unique strong solution. Moreover, this solution $X=(X_t)_{t\ge0}$ is a strong Markov process with the state space $(\C,\B(\C))$, see Proposition~\ref{P:MP} below. We denote the transition probabilities of $X$ by $P_t(x,\cdot)$, where $t\ge0$, $x\in\C$.

In this article we study the invariant probability measures of $X$. Further, we will drop the word ``probability'' and refer to these measures just as invariant measures.

It was shown in \cite[Theorems~3.1 and 3.7]{HMS11} (see also \cite[Section~6.1]{KS15}) that 
under \textbf{\ref{A:1}} $X$ has at most one invariant  measure and if it has one, then the transition probabilities weakly converge to this measure. Note however that \textbf{\ref{A:1}} does not guarantee the \textit{existence} of the invariant measure of $X$. Indeed, the equation
\begin{equation*}
d X(t)=d W(t),\quad t\ge0
\end{equation*}
satisfies \textbf{\ref{A:1}} but does not have an invariant measure. Also assumption \textbf{\ref{A:1}} alone does not imply any bound on convergence rate, see \cite[Remark~3.4]{HMS11}.

We will provide two different sets of conditions for the existence of an invariant measure for SFDE \eqref{mainSDDE} and present upper bounds for the rate of convergence to the equilibrium.  To formulate our results we need to introduce some notation.

Let $(E,\mathcal{B}(E))$ be a Polish space.
Recall that the \textit{Wasserstein }(or \textit{Kantorovich})\textit{ distance} between two probability measures $\mu$, $\nu$ on $(E,\mathcal{B}(E))$ is defined as follows:
\begin{equation*}
W_d(\mu,\nu):=\inf \E d (X,Y),
\end{equation*}
where $d$ is a lower semicontinuous metric on $E$ and the infimum is taken over all random variables $X$, $Y$ that are distributed as $\mu$ and $\nu$, correspondingly.
If the metric $d$ is the discrete metric, that is $d(x,y)=\I (x\neq y)$, then the Wasserstein distance is equivalent to the \textit{total variation distance} which is defined by
\begin{equation*}
d_{TV}(\mu,\nu):=\inf \P (X\neq Y)=\sup_{A\in\mathcal{B}(E)}|\mu(A)-\nu(A)|,
\end{equation*}
where again the infimum is taken over all random variables $X$, $Y$ that are distributed as $\mu$ and $\nu$, correspondingly. In the paper we will consider only bounded distances $d$. In this case, convergence in total variation implies convergence in the Wasserstein metric; the latter is also equivalent to the weak convergence  (see, e.g., \cite{Bog}).

Throughout the paper, we will take the space $\C$ as the state space $E$.
For $x \in \C$ we denote the diameter of the range of $x$ by
\begin{equation*}
D(x):=\sup_{t_1,t_2\in[-r,0]}|x(t_1)-x(t_2)|.
\end{equation*}

As in \cite[Section~5]{HMS11}, we consider the following family of distances on $\C$:
\begin{equation*}
d_\rho(x,y):=\frac{\|x-y\|}\rho\wedge1,\quad x,y\in\C,
\end{equation*}
where $\rho>0$.

\smallskip

Now we are in position to present our main results. We consider two different groups of conditions which are sufficient for 
the existence of invariant measure and exponential or subexponential convergence to the equilibrium.

\begin{Assumption}[Exponential convergence]\label{A:2}
The diffusion $g$ is globally bounded and the drift $f$ is sublinear. The latter means that there exist constants $\beta\in[0,1)$, $C>0$ such that
\begin{equation}\label{bound}
|f(x)| \le C(1+\|x\|^\beta), \quad x\in\C.
\end{equation}
Furthermore, there exist constants $\sigma,M>0$ and a  function $\kappa:\R_+\to \R_+$ such that
$\lim_{z \to \infty}\big(\kappa(z)z^{-\beta}\big)=\infty$ and
\begin{equation}\label{innerproduct}
\langle f(x),x(0)\rangle \le -\sigma|x(0)|,\quad  \text{for any } x\in\C \text{ with }D(x)\le \kappa(|x(0)|) \text{ and } |x(0)|\ge M.
\end{equation}
\end{Assumption}

\begin{Assumption}[Subexponential convergence]\label{A:3}
The diffusion $g$ and drift $f$ are globally bounded. Furthermore, there exist $\alpha\in(0,1)$, $\sigma>0$, $M>0$ and a  function $\kappa:\R_+\to \R_+$ such that $\lim_{t \to \infty}\big(\kappa(z)/\sqrt{\log z }\big)=\infty$ and
\begin{equation}\label{cond}
\langle f(x),x(0)\rangle \le - \sigma |x(0)|^{\alpha},\quad  \text{for any } x\in\C \text{ with }D(x)\le \kappa(|x(0)|) \text{ and } |x(0)|\ge M.
\end{equation}
\end{Assumption}

We will also present  results concerning convergence in the total variation distance. To state these results we need an additional assumption on the structure of the drift and the diffusion.

\begin{Assumption}[Convergence in total variation]\label{A:4}
The drift $f$ is globally Lipschitz and the diffusion $g$ depends on $x$ only through $x(0)$.
\end{Assumption}

\begin{Theorem}\label{T:1}
Suppose that Assumptions \textbf{\ref{A:1}} and \textbf{\ref{A:2}} hold. Then SFDE \eqref{mainSDDE} has a unique invariant measure $\pi$ and
the transition probabilities $P_t(x,\cdot)$ converge to it exponentially in the Wasserstein metric. That is, for any $\rho>0$ there exist $C>0$, $\lambda_1>0$, $\lambda_2>0$ such that for all $x\in\C$ we have
\begin{equation}\label{expwas}
W_{d_\rho}(P_t(x,\cdot),\pi)\le C e^{\lambda_1 |x(0)|+ D(x)}e^{-\lambda_2 t},\quad t\ge0.
\end{equation}
Moreover, if additionally Assumption \textbf{\ref{A:4}} holds, then the convergence in the Wasserstein metric
in \eqref{expwas} can be replaced by convergence in total variation metric.
\end{Theorem}

It is interesting to compare the obtained theorem with the corresponding result for SDEs. Recall that in the non-delay case the following condition is sufficient \cite{Ver88} for existence and uniqueness of the invariant measure and exponential convergence of transition probabilities in total variation:
\begin{equation}\label{sdecond}
\langle f(y), y\rangle\le -\sigma|y|,\quad |y|\ge M, y\in\R^d,
\end{equation}
where $M>0$, $\sigma>0$. In other words, for large enough $y\in\R^d$ the drift $f$ should point towards the origin.
Therefore, condition \eqref{innerproduct} is a direct equivalent of \eqref{sdecond} for SFDEs. We can call it
the \textit{extended Veretennikov--Khasminskii condition}.

Note that it is sufficient  to check \eqref{innerproduct} only for trajectories $x$ with not too large diameters. This is quite  important as it makes verifying \eqref{innerproduct} in practice much easier, see Section~3. The intuition here is the following. As one can see from the results of Section~4 below, for large enough $n$ with high probability $D(X_n)$ is approximately of the size $O (|X(n)|^\beta)$ regardless of the initial conditions. Thus, it is very unlikely that the trajectory will have a much bigger diameter. Even if it happens, one can just wait till the trajectory has a smaller diameter and then the drift would point towards the origin. Thus, one has to check the extended Veretennikov--Khasminskii condition only for ``typical'' trajectories. Note that this additional assumption $\lim_{z \to \infty}\big(\kappa(z)z^{-\beta}\big)=\infty$ is optimal, see Section~3 for  counterexamples.

The convergence in the Wasserstein metric in \eqref{expwas} cannot be replaced by the convergence in total
variation without additional Assumption \textbf{\ref{A:4}}. This is due to the reconstruction property discussed above. If the diffusion does not depend on the past, then SFDE does not have the reconstruction property and the convergence occurs in total variation.

Let us also mention that one cannot hope to replace \eqref{innerproduct}  by something like
$$
\langle f(x), x(-1)\rangle\le -\sigma|x(-1)|,\quad |x(-1)|\ge M.
$$
Indeed, the delayed Ornstein--Uhlenbeck equation \eqref{OU} satisfies this assumption, but it does not have an invariant measure.

Let us move on to our second main result that concerns subgeometrical convergence.
\begin{Theorem}\label{T:2}
Suppose that Assumptions \textbf{\ref{A:1}} and \textbf{\ref{A:3}} hold. Then SFDE \eqref{mainSDDE} has a unique invariant measure $\pi$ and
the transition probabilities $P_t(x,\cdot)$ converge to it subexponentially in the Wasserstein metric. That is, for any $\rho>0$ there exist $C>0$, $\lambda_1>0$, $\lambda_2>0$ such that for all $x\in\C$ we have
\begin{equation}\label{subexpwas}
W_{d_\rho}(P_t(x,\cdot),\pi)\le C e^{\lambda_1 |x(0)|^\alpha+\lambda_1 D(x)^2}e^{-\lambda_2 t^{\alpha/(2-\alpha)}},\quad t\ge0.
\end{equation}
Moreover, if additionally Assumption \textbf{\ref{A:4}} holds, then the convergence in the Wasserstein metric in \eqref{subexpwas} can be replaced
by convergence in total variation metric.
\end{Theorem}

We see that in the subgeometrical case it is also enough to check the extended
Ve\-re\-ten\-ni\-kov--Khasminskii condition only for trajectories with not too big diameter. The explanation is the same. It is worth mentioning that since the drift $f$  ``pushes'' to the origin weaker than in the exponential case, one has to check \eqref{cond} for a slightly bigger set of $x$ than just ``typical trajectories''.

We also would like to mention that the obtained rate of convergence to infinity in the right-hand side of
\eqref{subexpwas} matches the corresponding rate for the SDE case. The latter cannot be improved, see \cite[Section~7.1]{H10}.

The proofs of Theorems~\ref{T:1} and \ref{T:2} are postponed till Section 4.

\medskip

\textbf{Convention on constants.} Throughout the paper, we denote by $C$ a positive constant
whose value may change from line to line.

\section{Examples and Counterexamples}\label{S:3}

In this section we present a number of examples showing how the theoretical results from Section~\ref{S:2} can be used for studying convergence of SFDEs. In addition to it, we provide
some counterexamples that show the optimality (in a certain sense) of Assumptions \textbf{\ref{A:2}} and \textbf{\ref{A:3}}.

We begin with the following example.
\begin{Example}\label{E:31}
Let $d=m=1$. Consider an equation
\begin{equation}\label{SFDE1}
dX(t)=h(X(t-r)) dt + g (X_t)dW(t),\quad t\ge0,
\end{equation}
where the memory  $r\ge0$, $h\colon\R\to\R$ is a smooth function such that $h(z)=-|z|^\gamma\sign z $ for $|z|\ge1$, 
$\gamma\in(-1,1)$ and the diffusion $g$ is bounded Lipschitz and non-degenerate.

Clearly, Assumption \textbf{\ref{A:1}} holds and hence equation \eqref{SFDE1} has a unique strong solution. Let us check \textbf{\ref{A:2}} and \textbf{\ref{A:3}}. 
Put $\kappa(z):=z^{(1+\gamma)/2}$. Note that there exists large enough $M_0$
such that for any $x\in\C$ with $D(x)\le \kappa(|x(0)|)$ and  $|x(0)|\ge M_0$ we have
\begin{equation}\label{sootn}
|x(-r)|\ge |x(0)|-D(x)\ge |x(0)|-|x(0)|^{(1+\gamma)/2}\ge1
\end{equation}
and hence $h(x(-r))=-\sign(x(-r))|x(-r)|^\gamma$. 

Take now any $x\in\C$ with $D(x)\le \kappa(|x(0)|)$ and $|x(0)|\ge M_0$. Using \eqref{sootn}, we derive
\begin{align*}
h(x(-r))x(0)&=-\sign(x(-r))|x(-r)|^\gamma x(0)\\
&= -\sign(x(-r))|x(-r)|^\gamma (x(-r)+x(0)-x(-r))\\
&\le -|x(-r)|^{\gamma+1}+|x(-r)|^{\gamma}D(x)\\
&\le -\bigl|x(0)-|x(0)|^{(1+\gamma)/2}\bigr|^{\gamma+1}+(|x(0)|^\gamma+D(x)^\gamma+1) D(x)\\
&\le -\bigl|x(0)-|x(0)|^{(1+\gamma)/2}\bigr|^{\gamma+1}+3|x(0)|^{(1+\gamma)/2 + (\gamma\vee0)}.
\end{align*}
This implies that there exists large enough $M>M_0$, such that if $|x(0)|\ge M$ and $D(x)\le \kappa(|x(0)|)$, then 
\begin{equation*}
h(x(-r))x(0)\le -\frac12|x(0)|^{\gamma+1}.
\end{equation*}
Now if $\gamma\in[0,1]$, then \textbf{\ref{A:2}} holds. Therefore, by Theorem~\ref{T:1} SFDE \eqref{SFDE1} has a unique invariant measure and converges to it exponentially in the Wasserstein metric.

If $\gamma\in(-1,0)$, then \textbf{\ref{A:3}} holds. In this case we apply Theorem~\ref{T:2}. We obtain that \eqref{SFDE1} still has a unique invariant measure but converges to it subexponentially with the rate given in \eqref{subexpwas}. 
\qed
\end{Example}

\begin{Remark} In Example~\ref{E:31} it was crucial that it was sufficient to check condition \eqref{innerproduct} or \eqref{cond}  only for $x\in \C$ with not ``too large'' diameter. Evidently, these conditions  are not satisfied for \textbf{all} $x\in\C$. Thus, the exponential/subexponential ergodicity of \eqref{SFDE1} cannot be obtained by \cite[Remark~5.2]{HMS11} or 
\cite[Theorem~3.3]{Bu14}.
\end{Remark}

\begin{Example}\label{E:33}
Using  the same method we can study more general equations. Let $d,m\in\N$, $r\ge0$. We are interested in ergodic properties of the SFDE
\begin{equation}\label{boleeslozh}
dX(t)=h\Bigl(\int_{-r}^0 X(t+s) \mu(ds)\Bigr)dt + g (X_t)dW(t),\quad t\ge0,
\end{equation}
where $h:\R^d\to\R^d$ is a smooth function with $h(z):=-z|z|^{\gamma-1}$ for $|z|>1$, $\gamma\in(-1,1)$; $\mu$ is a finite signed measure with $\mu([-r,0])>0$; $g$ is as in Example~\ref{E:31}. We see that the drift and diffusion of equation \eqref{boleeslozh} satisfy Assumption \textbf{\ref{A:1}} and thus this equation has a unique strong solution. 

In order to verify \textbf{\ref{A:2}} and \textbf{\ref{A:3}}, we  choose  again $\kappa(z):=z^{(1+\gamma)/2}$. 
We consider the Jordan decomposition of the measure $\mu$:
$$
\mu=\mu^+-\mu^-,
$$
where $\mu^+$ and $\mu^-$ are two finite nonnegative measures, and note that for any $x\in\C$ with $D(x)\le \kappa(|x(0)|)$
we have 
\begin{align*}
\Bigl| \int_{-r}^0 x(s) \mu(ds)\Bigr|\ge& \Bigl| \int_{-r}^0 x(s) \mu^+(ds)\Bigr|- \Bigl| \int_{-r}^0 x(s) \mu^-(ds)\Bigr|\\
\ge& \mu^+([-r,0])|x(0)|- \int_{-r}^0 |x(s)-x(0)| \mu^+(ds)\\
&-\mu^-([-r,0])|x(0)|- \int_{-r}^0 |x(s)-x(0)| \mu^-(ds)\\
\ge& |x(0)|\mu([-r,0])-D(x)(\mu^+([-r,0])+\mu^-([-r,0]))\\
\ge& c_1|x(0)|-c_2|x(0)|^{(1+\gamma)/2},
\end{align*}
where $c_1=\mu([-r,0])$ and $c_2=(\mu^+([-r,0])+\mu^-([-r,0]))$. By our assumptions, we have $c_1>0$. 
The verification of \textbf{\ref{A:2}} and \textbf{\ref{A:3}} is completed exactly as in Example~\ref{E:31}. Thus, applying Theorems~\ref{T:1} and \ref{T:2}, we obtain that $X$ has a unique invariant measure and converges to it exponentially if $\gamma\in[0,1)$ or subexponentially if
$\gamma\in[-1,0)$.\qed
\end{Example}


\medskip
Now we move on and present some counterexamples to demonstrate a certain optimality of the conditions in Theorems~\ref{T:1} and \ref{T:2}.

First, we consider the case $\beta=0$. The next example shows that in this case it may happen that 
no invariant measure exists if the drift and diffusion satisfy all the conditions of Theorem~\ref{T:1} with the only exception that the condition $\lim_{z \to \infty}\kappa(z) = \infty$ in Assumption \textbf{\ref{A:2}} is replaced by $\liminf_{z \to \infty}\kappa (z) \ge N$, where $N>0$ is an arbitrarily large constant.

\begin{Example}\label{E:32}
Let $d=m=1$, $r=2$, $N>0$. Put $\kappa(z):=N$, $z\ge0$. Consider an equation
$$
dX(t)=f(X_t)dt + dW(t),\quad t\ge0
$$
where $f$ is a Lipschitz continuous function which takes values in $[-1,A]$, and satisfies $f(x)=A$ whenever $D(x)\ge N +1$
and $\langle f(x),x(0)\rangle \le - |x(0)|$ whenever $D(x) \le N$ and $|x(0)|\ge 1$. We claim that
$A>0$ can be chosen in such a way that for every fixed initial condition we have
\begin{equation}\label{limhotim}
\lim_{t \to +\infty} X(t)=+\infty\quad \text{a.s.}
\end{equation}
This would imply in particular that $X$ does not have an invariant measure.

To verify the claim we fix the initial condition $x\in\C$ and introduce an auxiliary sequence
\begin{equation*}
Y(n):=x(0)-n+W(n)+ A \sum_{i=1}^{n-1}  \I\bigl(W(i)-W(i-1)\ge N+2\bigr),\quad n\in\Z_+.
\end{equation*}
Since the drift $f$ is bounded from below by $-1$, we derive for $n\in\Z_+$, $n\ge1$
\begin{align}\label{diffXY}
X(n+1)-X(n)&=W(n+1)-W(n)+\int_{n}^{n+1} f(X_t)dt\nn\\
&\ge W(n+1)-W(n)-1+A\I \bigl(\inf_{t\in[n,n+1]}D(X_t) \ge N+1)\nn\\
&\ge W(n+1)-W(n)-1+A\I \bigl(X(n)-X(n-1)\ge N+1\bigr)\nn\\
&\ge W(n+1)-W(n)-1+A\I \bigl(W(n)-W(n-1)\ge N+2\bigr)\nn\\
&=Y(n+1)-Y(n),
\end{align}
where we also used the fact that the memory $r=2$ and hence $X(n)-X(n-1)\ge N+1$ implies that $D(X_t)\ge N+1$ for all $t\in[n,n+1]$.
Recall that by definition $X(0)=x(0)=Y(0)$ and $X(1)\ge x(0)-1+W(1)=Y(1)$. Therefore \eqref{diffXY} implies $X(n)\ge Y(n)$ for any $n\in\Z_+$.

By the strong law of large numbers,
\begin{equation*}
Y(n)/n\to -1 + A \P(\xi\ge N+2),\quad\text{a.s. whenever $n\to\infty$},
\end{equation*}
where $\xi$ denotes the standard Gaussian random variable. Hence by taking large enough $A$ we get $Y(n)/n\to 1$, a.s. as $n\to\infty$. Since $X(n)\ge Y(n)$, this yields
\eqref{limhotim}; thus the claim is proved and the process $X$ does not have an invariant measure.\qed
\end{Example}

Next we consider the case $\beta\in(0,1)$. We show that the condition $\lim_{z \to \infty}\kappa(z)z^{-\beta}=\infty$
in \textbf{\ref{A:2}} cannot be replaced with the condition  $\liminf_{z \to \infty}\kappa (z)z^{-\beta} \ge N$. Since we need to construct an example 
with unbounded  $\kappa$,  the proof here will be different from the proof in Example~\ref{E:32}.

\begin{Example} Let $\beta\in(0,1)$, $N>1$. Put $\kappa(z):=(N-1) z^\beta$, $z\ge0$. Consider the following stochastic delay equation for $d=m=1$, $r=2$:
\begin{equation}\label{sfde33}
d X(t)=f(X_t)dt + dW(t),  
\end{equation}
where $f$ is a Lipschitz continuous function such that $f(x)=5 N x(0)^\beta$ if $x(0) \ge 1$ and  $D(x) \ge  N x(0)^\beta$. Similar to Example~\ref{E:32}, one can easily extend $f$ in such a way that
Assumptions \textbf{\ref{A:1}} and \textbf{\ref{A:2}} hold with the only exception that condition \eqref{innerproduct} is satisfied for all $x\in \C$ with $D(x) \le (N-1) |x(0)|^\beta$ and $|x(0)| \ge 1$. Let us prove that SFDE \eqref{sfde33} does not have an invariant measure. 

Put $z_0:=(2N)^{1/(1-\beta)}$. We need the following technical statement.
\begin{Lemma}\label{L:technolemma}
Let $x\in \C$ be such that for some $t_1, t_2\in[-2,0]$ we have $x(t_2)\ge z_0$ and 
\begin{equation*}
x(t_1)\ge x(t_2)+2 N x(t_2)^\beta.
\end{equation*}
Then $D(x)\ge N |x(0)|^\beta$.
\end{Lemma}
\begin{proof}
We consider two different cases. If $x(0)<0$, then
\begin{equation*}
D(x)\ge x(t_2)-x(0)\ge N^{1/(1-\beta)}+|x(0)|= N|x(0)|^\beta \bigl(\frac{N^{\beta/(1-\beta)}}{|x(0)|^\beta}+
\frac{|x(0)|^{1-\beta}}N\bigr).
\end{equation*}
If now $|x(0)|^{1-\beta}>N$, then by above $D(x)\ge N|x(0)|^\beta$. If $|x(0)|^{1-\beta}\le N$, then 
$N^{\beta/(1-\beta)}\ge |x(0)|^\beta$ and, again, by above $D(x)\ge N|x(0)|^\beta$.

If  $x(0)\ge0$ we derive 
\begin{align*}
N x(0)^\beta&\le N (x(0)\vee x(t_1))^\beta\\
&\le N ((x(0)\vee x(t_1))-x(t_2))^\beta+ N x(t_2)^\beta\\
&\le ((x(0)\vee x(t_1))-x(t_2))\bigl(\frac{N}{(x(t_1)-x(t_2))^{1-\beta}}+\frac12\bigr)\\
&\le D(x)\bigl(\frac{N}{(2N)^{1-\beta}z_0^{\beta(1-\beta)}}+\frac12\bigr)\\
&=D(x).\qedhere
\end{align*}
\end{proof}

Now we go back to our equation \eqref{sfde33}. Define the ``bad'' set 
\begin{equation*}
G:=\{x\in\C: x(-1)\ge z_0 \text{\,\, and \,\,} x(0)\ge x(-1)+2 N x(-1)^\beta\}.
\end{equation*}
Let us prove that if the process $X$ starts with any initial condition from $G$, then it tends to infinity with positive probability.

Put $\tau:=\inf\{t\ge0: X(t)=1\}$ and $W_*:=\inf_{s\in[0,1]} W(s)$. Note that if $X_0\in G$, then, thanks to Lemma~\ref{L:technolemma}, we have $D(X_s)\ge N|X(s)|^\beta$ for any $s\in[0,1]$. Hence, it follows from the definition of $f$ that 
\begin{equation*}
\P_x\bigl(f(X_{s\wedge\tau})=5 N X(s\wedge\tau)^\beta \text{\,\,for every $s\in[0,1]$}\bigr)=1,
\end{equation*}
for any $x\in G$. This and \eqref{sfde33} imply that if $X_0\in G$, then for any $s\in[0,1]$ we have 
\begin{equation}\label{estxs}
X(s\wedge\tau)\ge X(0)+W(s\wedge\tau)\ge X(0)+W_*. 
\end{equation}
Therefore on the set  $\{W_*\ge -X(0)+1\}$ we have $\tau\ge1$. We employ this observation together with 
\eqref{estxs} to deduce for any $x\in G$ 
\begin{align*}
\P_x(X_1\in G)&=\P_x\bigl(X(1)\ge x(0)+2 N x(0)^\beta\bigr)\\
&\ge\P_x\bigl(X(1)\ge x(0)+2 N x(0)^\beta,\,\, W_*\ge -x(0)^\beta/2\bigr)\\
&\ge\P_x\bigl(5N(x(0)+W_*)^\beta+W_*\ge 2N x(0)^\beta,\,\, W_*\ge -x(0)^\beta/2\bigr)\\
&=\P(W_*\ge -x(0)^\beta/2\bigr)\\
&\ge 1- 2\exp\{-x(0)^{2\beta}/8\},
\end{align*}
where in the fourth transition we used the fact that $x(0)\ge 1$ and hence
$$
5N(x(0)+W_*)^\beta+W_*\ge5N(x(0)-x(0)^\beta/2)^\beta-x(0)^\beta/2\ge x(0)^\beta(5N-1)/2\ge 2 N x(0)^\beta,
$$
whenever $W_*\ge - x(0)^\beta/2$.

We apply the Markov property of $X$ to get for any $x\in G$
\begin{equation}\label{finalest}
\P_x(X_n\in G \text{ for all $n\in\Z_+$})\ge 1- 2\sum_{n=0}^\infty\exp\{-y_n^{2\beta}/8\},
\end{equation}
where we defined recursively $y_0:=x(0)$ and $y_n:=y_{n-1}+2N y_{n-1}^\beta$. Since $\beta>0$ and $y_n\ge x(0)+n$, we see that there exists large enough $Z_0\ge0$ such that the right--hand side of \eqref{finalest} is positive whenever $x(0)\ge Z_0$. Thus for any $x\in G':=G\cap \{x(0)\ge Z_0\}$ we have $\P_x(\lim_{n\to\infty}X(n)=+\infty)>0$. This implies by \cite[Theorem~3a and 3c]{Sch84} that $X$ does not have an invariant measure.\qed
\end{Example}

\begin{Example}
Finally, let us mention that the condition that the diffusion $g$ depends on $x$ only through $x(0)$ in Assumption~\textbf{\ref{A:4}} also cannot be dropped. Indeed, consider again SFDE \eqref{SFDE1} with $\gamma=0$, 
$g(x)=\wt g(x(-r))$, $r>0$, $x\in\C$ and $\wt g\colon\R\to\R_+$ is a bounded increasing and strictly positive function. This equation satisfies Assumptions \textbf{\ref{A:1}} and \textbf{\ref{A:2}} and its drift is Lipschitz. Nevertheless, as shown in \cite{Sch05} this equation converges to its invariant measure only weakly and not in total variation. Hence without this additional assumption, one cannot replace convergence in the Wasserstein metric in \eqref{expwas} and \eqref{subexpwas} by the convergence in total variation.
\end{Example}

\section{Proofs of the Theorems~\ref{T:1} and \ref{T:2}}
Till the end of this section without loss of generality and to simplify the notation we assume that the memory $r=1$. In Section~\ref{S:GT} we establish  general lemmas that are
useful for the proofs of our main results. In Sections~\ref{S:42} and \ref{S:43} we prove Theorems~\ref{T:1} and \ref{T:2}.

\subsection{General tools}\label{S:GT}

First let us verify that the strong solution to SFDE \eqref{mainSDDE} has indeed a  Markov property. Whilst this statement is well--known for the case of  Lipschitz drift and diffusion,
we were not able to find in the literature the proof of the Markov property of SFDE in the case of the one--sided Lipschitz drift. Thus we provide it here for the sake of completeness.

\begin{Proposition}\label{P:MP}
Suppose that Assumption \textbf{\ref{A:1}} holds. Then the unique strong solution to \eqref{mainSDDE} $X=(X_t)_{t\ge0}$ is a strong Markov process with the state space $(\C,\B(\C))$.
\end{Proposition}
\begin{proof}
We establish the Markov property using the standard technique (see, e.g., \cite[proof of Proposition~3.4]{RRG}). The authors are grateful to Alexei Kulik for communicating the main idea of the proof.

Fix $t\ge s\ge0$ and a bounded measurable function $f\colon\C\to\R$. Introduce the filtration $\F_r:=\sigma(W(u),\, 0\le u\le r)\vee \mathcal N$, where $r\ge0$ and $\mathcal N$ denotes the collection of null-sets in $\F$. Similarly, put $\G_{r,s}:=\sigma(W(u)-W(s),\, s\le u\le r)\vee \mathcal N$, $r\ge s$. Our goal is to show that
\begin{equation}\label{goal}
\E(f(X_t)|\F_s)=\E(f(X_t)|X_s).
\end{equation}
Since the function $f$ is arbitrary, \eqref{goal} would imply the Markov property for $X$.

To establish \eqref{goal} consider the equation
\begin{equation}\label{SFDEshift}
X^{(s,x)}(r)=x(0)+\int_s^r b(X^{(s,x)}_u)du +\int_s^r \sigma(X^{(s,x)}_u)dW(u),\quad r\ge s,\, x\in\C.
\end{equation}
It follows from \cite[Theorem~2.3]{RS08} and a simple shift argument that for each fixed $x\in\C$, equation \eqref{SFDEshift} has a unique strong solution and $X^{(s,x)}_t$ is a $\G_{t,s}|\,\B(\C)$
measurable function.

Introduce now a function $\Phi\colon\C\times\Omega\to\C$, $(x,\omega)\mapsto X^{(s,x)}_t(\omega)$. By above, for any fixed $x\in\C$ the function $\Phi(x,\cdot)$ is $\G_{t,s}|\,\B(\C)$--measurable.
By \cite[Proposition~5.4]{HMS11}, there exists $C>0$ such that
\begin{equation}\label{momentbound}
\E \|{X^{(s,x)}_t}-{X^{(s,y)}_t}\|^4\le e^{C(1+t)^2} \|x-y\|^4,\quad x,y\in \C.
\end{equation}
Therefore $\Phi(x,\cdot)$ is continuous in probability with respect to $x$. Since the space $\C$ is Polish, \cite[Theorem~3.1]{Mnogo}
implies that $\Phi$ has a modification $\wt \Phi$ that is $(\B(\C)\otimes\G_{t,s})|\,\B(\C)$--measurable.

Strong uniqueness of solutions to \eqref{SFDEshift} (\cite[Theorems~2.2 and 2.3]{RS08}) yields that
\begin{equation*}
X_t(\omega)=\wt \Phi(X_s,\omega)\quad  \text{a.s.},
\end{equation*}
where we also used the fact that $X_s$ is $\F_s$ measurable and the $\sigma$-algebras $\F_s$ and $\G_{t,s}$ are independent.

Now let us prove \eqref{goal}. It follows from the measurability properties of $\wt \Phi$ established above, that $f(\wt \Phi(X_s,\cdot))$ is $\sigma(X_s, \G_{t,s})$--measurable.
Using again the independence of $\F_s$ and $\G_{t,s}$ and a standard approximation argument (see, e.g., \cite[Theorem~7.1.2]{Oks}), we derive
\begin{equation*}
\E(f(X_t)|\mathcal F_s)=  \E(f(\wt \Phi(X_s,\cdot))|\mathcal F_s)=\E f(\wt \Phi(x,\cdot))|_{x=X_s}.
\end{equation*}
Similarly,
\begin{equation*}
\E(f(X_t)| X_s)= \E(f(\wt \Phi(X_s,\cdot))| X_s)= \E f(\wt \Phi(x,\cdot))|_{x=X_s}
\end{equation*}
and therefore identity \eqref{goal} holds.

To establish the strong Markov property we employ again bound  \eqref{momentbound}. This inequality and the Portmanteau theorem imply that the process $X$ is Feller. Since it has also
continuous trajectories, it is strongly Markov \cite[Theorem~3.3.1]{RY}.
\end{proof}

As mentioned above, our approach for establishing ergodicity is based on  Lyapunov functions. The propositions below  state that if one is able to construct a ``good'' Lyapunov function,
then SFDE~\eqref{mainSDDE} possesses  all the required ergodic properties. These propositions essentially
follow from the corresponding results in \cite{Bu14} and \cite{HMS11}.

Recall that by $P_t$ we denoted the Markov semigroup associated with the strong solution to \eqref{mainSDDE}.

\begin{Proposition}\label{P:41}
Suppose that Assumption \textbf{\ref{A:1}} holds. Suppose that there exists a measurable function $V\colon\C\to\R_+$ such that $\lim_{\|x\|\to\infty}V(x)=+\infty$ and
\begin{equation}\label{Lyap}
\E_x V(X_1)\le V(x)-\Psi(V(x))+C,\quad x\in\C,
\end{equation}
where $\Psi\colon\R_+\to(0,+\infty)$ is a differentiable concave function increasing to infinity. Then SFDE \eqref{mainSDDE} has a unique invariant measure $\pi$. Furthermore, for any $\rho>0$, $\eps>0$ there exist
constants $C_1>0$, $C_2>0$ such that
\begin{equation}\label{subexwastool}
W_{d_\rho}(P_t(x,\cdot),\pi)\le \frac{C_1(1+V(x))}{\Psi(H_\Psi^{-1}(C_2t))^{1-\eps}},\quad t\ge0,\,x\in\C.
\end{equation}
Here $H_\Psi(t):=\int_1^t \frac1{\Psi(s)}\,ds$, $t\ge0$, and $H_\Psi^{-1}$ is the inverse function.
\end{Proposition}
\begin{proof}
Fix $\rho>0$. It follows from \cite[Sections 5.1 and 5.2]{HMS11} that for some $n_0\in\N$, $\delta\in(0,\rho)$
we have
\begin{equation}\label{thanks1}
W_{d_\delta}(P_{n_0}(x,\cdot),P_{n_0}(y,\cdot))\le d_\delta(x,y),\quad x,y\in\C.
\end{equation}
and for any $N>0$ there exists $\gamma\in(0,1)$ such that
\begin{equation}\label{thanks2}
W_{d_\delta}(P_{n_0}(x,\cdot),P_{n_0}(y,\cdot))\le \gamma d_\delta(x,y),\quad x,y\in\C,\, \|x\|\le N,\, \|y\|\le N.
\end{equation}

Consider now an auxiliary skeleton Markov chain with the state space $(\C,d_\delta)$ and transition kernel
$$
\wt P(x,A):=P_{n_0}(x,A),\quad x\in\C,\, A\in\B(\C).
$$

Let us check that this chain satisfies all the conditions of \cite[Theorem~2.1]{Bu14}. By iterating
\eqref{Lyap} $n_0$ times, we see that
\begin{equation*}
\int_\C  V(y) \wt P(x,dy)=\E_x V(X_{n_0}) \le V(x)-\Psi(V(x))+n_0 C,\quad x\in\C.
\end{equation*}
Therefore the first condition of \cite[Theorem~2.1]{Bu14} holds. As explained above, the space $(\C,d_\delta)$
is a complete separable metric space, therefore the second condition is also met. It follows from estimates \eqref{thanks1}, \eqref{thanks2}, and our assumption $\lim_{\|x\|\to\infty}V(x)=+\infty$ that the third and the fourth conditions of
\cite[Theorem~2.1]{Bu14} are also satisfied.

Thus, all conditions of \cite[Theorem~2.1]{Bu14} are met. Hence the skeleton chain has a unique invariant measure $\pi$,
and there exist constants $C_1>0$, $C_2>0$ such that
\begin{equation*}
W_{d_\rho}(\wt P_n(x,\cdot),\pi)\le W_{d_\delta}(\wt P_n(x,\cdot),\pi)\le \frac{C_1(1+V(x))}{\Psi(H_\Psi^{-1}(C_2n))^{1-\eps}},\quad n\in\Z_+,\,x\in\C,
\end{equation*}
where we also used the fact that $d_\rho\le d_\delta$. Now by a standard argument (see, e.g., \cite[p.~550]{Bu14}), we see that the measure $\pi$ is also a unique invariant measure for our original Markov kernel
$P_t$ and that bound \eqref{subexwastool} holds.
\end{proof}

\begin{Proposition}\label{P:42}
Assume that all  conditions of Proposition~\ref{P:41} are met. Suppose additionally that Assumption
\textbf{\ref{A:4}} is satisfied.  Then the convergence in the Wasserstein metric in \eqref{subexwastool} can be replaced
by convergence in total variation metric.
\end{Proposition}

\begin{proof}
We begin by observing that, thanks to the additional Assumption \textbf{\ref{A:4}}, the Markov semigroup $P_t$ satisfies
the Harnack inequality. Namely, it follows from  \cite[Theorem~4.1]{WY11} (see also \cite[Theorem~1.1]{ERS}) that for any $t>1$ and large enough $p>p_0$,
there exists $C=C(p)$ such that
\begin{equation*}
\bigl(P_t(x,A)\bigr)^p\le P_t (y,A)e^{C (1+\|x-y\|^2)},\quad x,y\in \C,\,\, A\in\B(\C).
\end{equation*}
Therefore for any $x,y\in\C$, $A\in\B(\C)$ we have
\begin{align*}
P_t(x,A)-P_t (y,A)&\le P_t(x,A)-\bigl(P_t(x,A)\bigr)^p (e^{-C (1+\|x-y\|^2)}\wedge p^{-1})\\
&\le 1- (e^{-C (1+\|x-y\|^2)}\wedge p^{-1}).
\end{align*}
Thus, we have the following bound on the total variation distance.
\begin{equation*}
d_{TV}(P_t(x,\cdot),P_t(y,\cdot))\le 1- (e^{-C (1+\|x-y\|^2)}\wedge p^{-1}),\quad x,y\in\C.
\end{equation*}

Now similar to the proof of Proposition~\ref{P:41} we fix $t=2$ and consider the skeleton Markov chain with the transition kernel
$$
\wt P(x,A):=P_{2}(x,A),\quad x\in\C,\, A\in\B(\C).
$$
It follows from the above that $\wt P$ satisfies all the assumptions of \cite[Theorem~1.15 and Theorem~1.13]{Kuni15}.
Note that we do not have to check that the skeleton Markov chain is irreducible or aperiodic, see also the related discussion in \cite[Remark~3.3]{HM11}.

Thus, if $\pi$ denotes the invariant measure of $P$ (its existence and uniqueness was already established in Proposition~\ref{P:41}),
then by \cite[Theorem~1.15 and Theorem~1.13]{Kuni15} we have
\begin{equation*}
d_{TV}(P_{2n}(x,\cdot),\pi)=d_{TV}(\wt P_n(x,\cdot),\pi)\le \frac{C_1(1+V(x))}{\Psi(H_\Psi^{-1}(C_2n))^{1-\eps}},\quad n\in\Z_+,\,x\in\C.
\end{equation*}
Therefore if $t=2n+s$, where $n\in\Z_+$ and $s\in[0,2]$, then for any $x\in\C$ we derive
\begin{equation*}
d_{TV}(P_t(x,\cdot),\pi)=d_{TV}(P_{2n+s}(x,\cdot),P_s\pi)\le d_{TV}(P_{2n}(x,\cdot),\pi)\le\frac{C_1(1+V(x))}{\Psi(H_\Psi^{-1}(C_2t))^{1-\eps}},
\end{equation*}
where we made use of the nonexpanding property of the total variation metric. This completes the proof of the proposition.
\end{proof}

The following two lemmas describe the behaviour of $D(X_t)$. These lemmas provide very important estimates
that will be used in the sequel.

\begin{Lemma}\label{L:51}
Suppose that Assumption \textbf{\ref{A:1}} holds. Assume that the drift $f$ satisfies the growth condition \eqref{bound}
with $\beta\in[0,1)$ and the diffusion $g$ is globally bounded. Then there exists a constant $C>0$ such that
for any $\lambda\ge0$ we have
\begin{equation}\label{radius}
\E_x e^{\lambda D(X_1)}\le e^{C\lambda(|x(0)|^\beta + D(x)^\beta+\lambda+1)},\quad x\in\C.
\end{equation}
Moreover, there exist $C>0$, $\lambda_0>0$ such that
\begin{equation}\label{firsteq}
\E_x e^{\lambda_0 D(X_1)^2}\le e^{C(|x(0)|^{2\beta} + D(x)^{2\beta}+1)},\quad x\in\C.
\end{equation}
Finally, there exist constants $C_1$, $C_2>0$ such that for any $z>0$ we have
\begin{equation}\label{secondeq}
\P_x(D (X_1)\ge z)\le C_1 e^{C_1|x(0)|^{2\beta} + C_1D(x)^{2\beta}-C_2 z^2},\quad x\in\C.
\end{equation}
\end{Lemma}
\begin{proof}
We begin by observing that for any $x\in\C$
\begin{equation}\label{step1}
D(X_1)\le 2\sup_{0\le t \le 1}|X(t)-X(0)|\le 2\int_0^1 |f(X_s)|\,ds+2\sup_{0\le t\le 1} |M(t)|,
\end{equation}
where we denoted $M(t):=\int_0^t g (X_s)\,d W(s)$. We make use of the growth condition \eqref{bound} and the estimate
$D(X_s)\le D(x)+D(X_1)$, valid for all $s\in[0,1]$, to derive
\begin{align}\label{step1.5}
\int_0^1 |f(X_s)|\,ds&\le C\int_0^1(\|X_s\|^\beta+1)\,ds\le
C(|x(0)|^\beta+D(x)^\beta+D(X_1)^\beta+1)\nn\\
&\le \frac14 D(X_1)+C|x(0)|^\beta+C D(x)^\beta+C,
\end{align}
where we also used the fact that $\beta<1$ and hence for some $C_\beta>0$ one has $Cz^\beta\le z/4+C_\beta$ for all $z\ge0$. Substituting \eqref{step1.5}
into \eqref{step1}, we get
\begin{equation}\label{step2}
D(X_1)\le C |x(0)|^\beta+C D(x)^\beta+C\sup_{0\le t\le 1} |M(t)|+C.
\end{equation}
To estimate the exponential moments of $\sup_{0\le t\le 1} |M(t)|$ we use the Dambis--Dubins--Schwarz theorem and the global boundedness of $g$. It follows that
\begin{equation*}
M(t)= (B^{1}(\tau_1),\dots,B^d(\tau_d)),\quad t\in[0,1],
\end{equation*}
where $B^1, B^2,\dots,B^d$ are (possibly dependent) one-dimensional Brownian motions and
\begin{equation*}
\tau_i=\tau_i(t):=\int_0^t \sum_{j=1}^m (g^{ij} (X_s))^2\,d s\le C_g,\quad t\in [0,1], \, i=1,2,\hdots,d.
\end{equation*}
and the constant $C_g$ does not depend on $x$. Thus,
\begin{equation*}
\sup_{t \in [0,1]} |M(t)| \le  \sum_{i=1}^d\sup_{t \in [0,C_g]}|B^i(t)|. 
\end{equation*}
Therefore, we apply the Cauchy--Schwarz inequality to get for any $\lambda\ge0$, $x\in\C$
\begin{align*}
\E_x \exp\{\lambda \sup_{t \in [0,1]}|M(t)|\}&\le\E\exp\bigl\{\lambda
\sum_{i=1}^d\sup_{t \in [0,C_g]}|B^i(t)|\bigr\} \le\E\exp\{d\lambda \sup_{t \in [0,C_g]}|B(t)|\}\\
&=\E\exp\{C\lambda|B(1)|\}\le \exp\{C(\lambda + \lambda^2)\},
\end{align*}
where by $B$ we denoted a standard Brownian motion. This together with \eqref{step2} implies \eqref{radius}.

Arguing as above, we see that  there exists constants $C>0$, $\lambda_0>0$ such that for any $x\in\C$ we have
\begin{equation*}
\E_x \exp\{\lambda_0 \sup_{t \in [0,1]}M(t)^2\}<C.
\end{equation*}
This together with \eqref{step2} implies \eqref{firsteq}.

Estimate \eqref{secondeq} follows directly from \eqref{firsteq} and the Chebyshev inequality.
\end{proof}

\begin{Lemma}\label{L:52}
Suppose that the assumptions of Lemma~\ref{L:51} hold. Let $\psi\colon\R_+\to\R_+$ be an increasing continuous concave function such that $\psi(t)\le C_\psi(t+1)$ for some $C_\psi\ge1$ and any $t\in\R_+$.
Then there exist constants $C_1>0$, $C_2>0$ such that for any $t\in[0,1]$ and $x\in\C$ with
$D(x)\le \psi(|x(0)|)/( 4 C_\psi)$ we have
\begin{equation}\label{psiest}
\P_x(D (X_t)\ge \psi(|X(t)|))\le C_1 e^{C_1|x(0)|^{2\beta} + C_1 D(x)^{2\beta}-C_2 \psi(|x(0)|)^2}.
\end{equation}
\end{Lemma}
\begin{proof}
First, let us note that for any $t\in[0,1]$ we have $|X(t)|\ge|X(0)|-D(X_t)$. Hence
\begin{equation*}
\psi(|X(t)|)\ge \psi(|X(0)|)-\psi(D(X_t)).
\end{equation*}
Therefore, using the condition $\psi(s)\le C_\psi(s+1)$, $s\in\R_+$, we derive for any $t\in[0,1]$,
$x\in\C$ with $D(x)\le \psi(|x(0)|)/(4C_\psi)$
\begin{align*}
\P_x\bigl(D (X_t)\ge \psi(|X(t)|)\bigr)\le& \P_x\bigl((C_\psi+1)D (X_t)+C_\psi\ge \psi(|X(0)|)\bigr)\\
\le&\P_x\bigl(D (X_t)\ge \frac{\psi(|X(0)|)}{2C_\psi}-1\bigr)\\
\le&\P_x\bigl(D (X_1)\ge \frac{\psi(|X(0)|)}{4C_\psi}-1\bigr),
\end{align*}
where in the last inequality we used the assumption $D(x)\le \psi(|x(0)|)/(4C_\psi)$ and the estimate $D(X_t)\le D(X_1)+D(x)$. Now the application of estimate \eqref{secondeq} yields \eqref{psiest}.
\end{proof}

\subsection{Proof of Theorem~\ref{T:1}}\label{S:42}

To prove Theorem~\ref{T:1} we use the following Lyapunov function:
\begin{equation}\label{vau}
V(x):=\exp\big\{\lambda |x(0)|+ \bigl(D(x)-\gamma |x(0)|^\beta\bigr)_+\big\},\quad x\in\C,
\end{equation}
where the parameters $\lambda>0$, $\gamma>0$ are to be set later. To avoid technicalities we assume that the function $\kappa$
from Assumption \textbf{\ref{A:2}} is increasing and concave. Clearly, this is not a restriction at all: if Assumption \textbf{\ref{A:2}} is satisfied, then there exists an increasing concave function $\wt \kappa$ such that \textbf{\ref{A:2}} is also satisfied with $\wt \kappa$ in place of $\kappa$.  It follows that there exists a constant $C_\kappa\ge1$ such that $\kappa(t)\le C_\kappa(t+1)$ for any $t\ge0$.

First, let us prove that $V$ is a Lyapunov function on the set where $D(x)$ is relatively big compared with $|x(0)|$. The heuristics here is as follows. As we explained in the introduction, it is not typical for the process $X$ to have a large diameter. Thus,  $D(X_t)$ will decrease with high probability. This will also cause the decrease of the Lyapunov function $V$. Formally we have the following lemma.

\begin{Lemma}\label{three}
Suppose that Assumptions \textbf{\ref{A:1}} and \textbf{\ref{A:2}} hold. Let $V$ be the Lyapunov function defined in \eqref{vau}. Then for any  $\lambda>0$, $\gamma>0$, $\eps\in(0,1)$ there exist $L>0$ and $B>0$ such that
\begin{equation}\label{mainest3}
\E_x V(X_1) \le  \eps  V(x),
\end{equation}
for any $x\in \C$ such that $D(x)\ge L$ and $D(x)\ge B|x(0)|^\beta$.

Further, for any $\lambda>0$, $\gamma>0$, $L>0$, $R>0$ we have
\begin{equation}\label{lemmasup}
\sup_{\substack{x\in\C, |x(0)|\le R\\ D(x)\le L}}\E_x V(X_1)<\infty.
\end{equation}

\end{Lemma}
\begin{proof}
Fix $\lambda>0$,  $\gamma>0$. Then for any $x\in\C$ we have
\begin{align}\label{firstep}
\E_x V(X_1)&\le \E_x \exp\{\lambda |X(1)|+D(X_1)\}\nn\\
&\le \E_x \exp\{\lambda (|X(0)|+D(X_1))+ D(X_1)\}\nn\\
&=e^{\lambda |x(0)|}\E_x e^{(\lambda+1) D(X_1)},
\end{align}
which immediately yields \eqref{lemmasup}. 

To establish \eqref{mainest3}, we apply estimate \eqref{radius} to \eqref{firstep}. We get
\begin{align}\label{promest}
\E_x V(X_1)&\le \exp\big\{\lambda |x(0)|+C(\lambda+1)(|x(0)|^\beta + D(x)^\beta +\lambda+2)\}\nn\\
&\le V(x)\exp\big\{- D(x)+C_1|x(0)|^\beta+C_2D(x)^\beta +C_3\big\},
\end{align}
where $C_1:=\gamma +C(\lambda+1)$, $C_2:=C(\lambda+1)$, and
$C_3:=C(\lambda+1)(\lambda+2)$. Now take large enough $L_0=L_0(\lambda)$ such that $C_2 z^\beta+C_3\le  z/3$ whenever $z\ge L_0$. Put $B:=3C_1$.
Then it follows from \eqref{promest} that for any $x\in \C$ with $D(x) \ge L_0$ and $D(x) \ge B|x(0)|^\beta$ we get
\begin{equation*}
\E_x V(X_1)\le V(x)e^{- D(x)/3}\le V(x) e^{- L_0/3},
\end{equation*}
which implies \eqref{mainest3}.
\end{proof}

The case when the initial diameter is ``small'' is much more complicated and more precise estimates are needed. In this case $D(X_t)$ stays at the same level, and the decrease of the Lyapunov function $V$ happens due to the decrease of $|X(t)|$. To formalize these ideas we will use the following version of the Gronwall inequality.

\begin{Lemma}\label{L:Gron}
Let $T>0$, $\theta>0$, $r\ge0$. Let $f\colon[0,T]\to\R$ be a continuous function  satisfying for any
$0\le s\le u \le T$ the following inequality:
\begin{equation*}
f(u)\le f(s)-\int_s^u (\theta f(t)-r)dt.
\end{equation*}
Then for any $t\in[0,T]$
\begin{equation}\label{newGR}
f(t)\le e^{-\theta t}f(0)+r/\theta.
\end{equation}
\end{Lemma}
\begin{proof}
Consider the function $g(t):=e^{-\theta t}(f(0)-r/\theta)+r/\theta$, $0\le t\le T$. Clearly, $g(0)=f(0)$ and for any $0\le s\le u \le T$ we have
\begin{equation*}
g(u)= g(s)-\int_s^u (\theta g(t)-r)dt.
\end{equation*}
Hence by \cite[Proposition~9.2]{HM15}, we have $f(t)\le g(t)$ for any $t\in[0,T]$. This implies \eqref{newGR}.
\end{proof}

For $B>0$, $N>0$ define
\begin{equation}\label{CBN}
\C_{B,N}:=\{x\in\C: D(x)\le B |x(0)|^\beta \text{ and }|x(0)|\ge N\}.
\end{equation}
We start treating this case with the following key lemma.
\begin{Lemma}\label{six}
Suppose that the assumptions of Lemma~\ref{three} hold. Then there exist $\nu>0$, $\rho \in (0,1)$, such that for every $B>0$ there exists $N>0$ such that
\begin{equation*}
\E_x e^{\nu|X(1)|}\le e^{\nu|x(0)|}(1- \rho)
\end{equation*}
for all $x \in \C_{B,N}$.
\end{Lemma}
\begin{proof}
Recall the definition of constant $M$ from condition \eqref{innerproduct}. With such $M$ in hand, let $\lambda\in(0,1)$,  let $\phi\colon\R^d\to\R_+$ be a smooth function such that $\phi(y):=\exp(\lambda |y|)$ for $|y|\ge M$. We want to apply a version of Gronwall's lemma to the function $u \mapsto \E_x  \phi(X(u))$. 

First of all, we observe that this function is finite. Indeed, thanks to \eqref{radius}, we have
for any $x\in\C$, $0 \le  u \le 1$
$$
\E_x  \phi(X(u))\le C +\E_x e^{|X(u)|}\le C + e^{|x(0)|}\E_x e^{D(X_1)|}<\infty.
$$

We make use of assumption \textbf{\ref{A:2}} and apply Ito's lemma.
We have for $|y|\ge M$
\begin{align*}
&\frac{\d \phi (y)}{\d y_i}=\lambda e^{\lambda |y|}|y|^{-1}y_i;\\
&\frac{\d^2 \phi (y)}{\d y_i\d y_j}=\lambda^2 e^{\lambda |y|}|y|^{-2}y_iy_j-
\lambda e^{\lambda |y|}|y|^{-3}y_iy_j,\quad i\neq j;\\
&\frac{\d^2 \phi (y)}{\d y_i^2}=\lambda^2 e^{\lambda |y|}|y|^{-2}y_i^2-
\lambda e^{\lambda |y|}|y|^{-3}y_i^2+\lambda e^{\lambda |y|}|y|^{-1}.
\end{align*}
Thus, we derive for any $x\in\C$,
$0 \le s \le u \le 1$
\begin{align}\label{Ito}
\E_x  \phi(X(u))\le&\E_x  \phi(X(s)) -\sigma\lambda\int_s^u \E_x \big(\I_{\{D(X_t)\le \kappa(|X(t)|), |X(t)|>M\}}\phi(X(t))\big)\, d t \nn\\
&+
C\lambda  \int_s^u\E_x\big( \I_{\{D(X_t)> \kappa(|X(t)|), |X(t)|>M\}}\big(1+\|X_t\|^\beta\big) e^{\lambda |X(t)|}\big)\, d t \nn\\
&+
C\lambda^2  \int_s^u\E_x e^{\lambda |X(t)|}\,d t\nn\\
&+C\lambda  \int_s^u \E_x \I_{\{|X(t)|>M\}}e^{\lambda |X(t)|}|X(t)|^{-1}\,d t+C (u-s)+\E_x (M(u)-M(s))\nn\\
=:&\E_x  \phi(X(s))-\sigma\lambda I_1+C \lambda I_2+C\lambda^2 I_3\nn\\
&+ C \lambda I_4+C (u-s)+\E_x (M(u)-M(s)),
\end{align}
where we denoted
$$
M(u):=\int_0^u \sum_{i,j} \frac{\d \phi }{\d y_i}(X(t))g^{ij}(X_t)dW^j(t).
$$
We use the boundedness of $g$, the definition of $\phi$, and estimate~\eqref{radius} to derive
\begin{align*}
\E_x \langle M\rangle_u\le C + C \int_0^u \E_x e^{2|X(t)|}\,dt\le 
C + C e^{2|x(0)|}\E_x e^{2D(X_1)}<\infty.
\end{align*}
Thus $(M(t))_{t\ge0}$ is a martingale and 
\begin{equation}\label{martest}
\E_x (M(u)-M(s))=0. 
\end{equation}

To estimate $I_1$ we assume that $|x(0)|\ge M_0:=M +\kappa(M)$ in which case $|X(t)|<M$, $t  \in [0,1]$  implies
$D(X_t)\ge |X(t)-x(0)|> \kappa(M)\ge \kappa(|X(t)|)$. Therefore,
\begin{align}\label{etap}
\E_x& \Big(\I_{\{D(X_t)\le \kappa(|X(t)|),|X(t)|\ge M\}}\phi(X(t))\Big)\nn\\
&= \E_x \Big(\I_{\{D(X_t)\le \kappa(|X(t)|)\}}\phi(X(t))\Big)\nn\\
&= \E_x \phi(X(t)) - \E_x \big(\I_{\{D(X_t)> \kappa(|X(t)|)\}}\phi(X(t))\big)\nn \\
&\ge  \E_x \phi(X(t)) - \big(\P_x \big(D(X_t)> \kappa(|X(t)|)\big)\big)^{1/2}
\big(\E_x [\phi(X(t))^2]\big)^{1/2}.
\end{align}
We continue this calculation in the following way. Recall that $\kappa(z)/z^\beta\to+\infty$ as $z\to+\infty$. Therefore 
for any $B>0$ there exists $N=N(B)>0$ such that
$$
\C_{B,N}\subset \{x\in\C: D(x)\le \frac{\kappa(|x(0)|)}{4C_\kappa}\},
$$
where $C_\kappa$ was defined in the beginning of Section~\ref{S:42}. Hence Lemma \ref{L:52} implies that there exist constants $C_1$, $C_2$ such that for any $B>0$ there exists $N=N(B)>0$ such that
for any $t\in[0,1]$
\begin{equation}\label{1factor}
\P_x (D(X_t)> \kappa(|X(t)|))\le C_1e^{-C_2 \kappa(|x(0)|)^{2}},\quad x\in\C_{B,N}.
\end{equation}
By Lemma~\ref{L:51}, for any $x\in\C$, $t\in[0,1]$ we have
\begin{align*}
\E_x \phi(X(t))^2&\le C + \E_x e^{2\lambda|X(t)|}\\
&\le  C + e^{2\lambda|x(0)|} \E_x e^{2 D(X_1)}\\
&\le  C  + e^{2\lambda|x(0)|}e^{C(|x(0)|^\beta+|D(x)|^\beta+1)}.
\end{align*}
Using again the fact that $\kappa(z)/z^\beta\to+\infty$ and combining the above estimate with \eqref{etap} and
\eqref{1factor}, we see that there exist constants $C_3>0$, $C_4>0$ such that for any $B>0$ there exists $N_1=N_1(B)>0$ such that
\begin{equation}\label{I1}
I_1\ge \int_s^u \E_x \phi(X(t))\,dt-C_3e^{\lambda|x(0)|}e^{-C_4\kappa(|x(0)|)^{2}}(u-s),\quad x\in\C_{B,N_1}.
\end{equation}

Next, we estimate the integrand in  $I_2$. We estimate this term applying H\"older's inequality to the three factors. The first and third factors are estimated as above. Further, for any $x\in\C$, $t\in[0,1]$
\begin{align*}
\E_x (1+\|X_t\|^\beta)^3&\le \E_x(1+|x(0)|^\beta+D(x)^\beta+D(X_1)^\beta)^3\\
&\le C\bigl(1+|x(0)|^{3\beta}+D(x)^{3\beta} + \E_x [D(X_1)^{3\beta}]\bigr)\\
&\le  C\bigl(1+|x(0)|^{3\beta}+D(x)^{3\beta} + e^{C(|x(0)|^\beta+|D(x)|^\beta+1)}\bigr),
\end{align*}
where in the last inequality we used Lemma~\ref{L:51}. Thus, there exist constants $C_5>0$, $C_6>0$ such that for any $B>0$ there exists $N_2=N_2(B)>0$ such that
\begin{equation}\label{I2}
I_2\le C_5e^{\lambda|x(0)|}e^{-C_6\kappa(|x(0)|)^{2}}(u-s),\quad x\in\C_{B,N_2}.
\end{equation}

For $\widetilde M >M$, $x\in\C$ we estimate the integrand in $I_4$ as follows.
\begin{equation}\label{I4}
{\E_x} \big(\I_{\{|X(t)|>M|\}}\exp\{\lambda |X(t)|\}|X(t)|^{-1}\big)
\le  \widetilde M^{-1}   \E_x\phi(X(t)) +
\frac1M e^{\lambda \widetilde M}.
\end{equation}
Combining \eqref{martest}, \eqref{I1}, \eqref{I2} and \eqref{I4} with \eqref{Ito}, we see that there exist constants $C_7>0$, $C_8>0$ such that for any $B>0$ there exists $N_3=N_3(B)>0$ such that for $x\in \C_{B,N_3}$
\begin{align*}
\E_x  \phi(X(u))\le&\E_x  \phi(X(s)) -(\sigma\lambda -C\lambda^2-C\lambda\widetilde M^{-1})\int_s^u \E_x \phi(X(t))\,d t\\
&+ C_7(e^{\lambda |x(0)|}e^{-C_8\kappa(|x(0)|)^{2}}+e^{\lambda\widetilde M})(u-s).
\end{align*}
Now, we choose $\widetilde M$ large enough and $\lambda>0$ small enough so that
\begin{equation*}
\theta:=\sigma\lambda -C\lambda^2-C\lambda\widetilde M^{-1}>0.
\end{equation*}
Clearly $\theta$ is independent of $B$ and $N_3$. Recall also that we have checked in the beginning of the proof that 
$\E_x  \phi(X(u))<\infty$ for any $u\in[0,1]$. Thus, by Lemma~\ref{L:Gron} (a version of the Gronwall inequality) we get
\begin{equation*}
\E_x e^{\lambda|X(1)|}\le  \E_x\phi(X(1))+e^{\lambda M} \le e^{\lambda|x(0)|}e^{-\theta}+
e^{\lambda|x(0)|}\zeta(|x(0)|),\quad x\in \C_{B,N_3},
\end{equation*}
where the function $\zeta(z)$ is independent of $B$ and $N_3$, and tends to $0$ as $z\to\infty$. This implies
the statement of the lemma.
\end{proof}

Now we are able to establish a crucial Lyapunov inequality. Recall the definition of $\C_{B,N}$ in \eqref{CBN}.

\begin{Lemma}\label{L:4.8}
Suppose that Assumptions \textbf{\ref{A:1}} and \textbf{\ref{A:2}} hold. Then there exist $\lambda>0$, $\gamma>0$,
$c_1\in(0,1)$, $c_2>0$ such that the function $V$ defined in \eqref{vau} satisfies the following inequality:
\begin{equation}\label{statement4.8}
\E_x V(X_1)\le (1-c_1) V(x)+c_2,\quad x\in\C.
\end{equation}
\end{Lemma}
\begin{proof}
First we note that for any $\lambda>0$, $\gamma>0$, $x\in\C$  we have
\begin{equation}\label{secondfactor}
\E_xV(X_1)\le \big(\E_x e^{2\lambda|X(1)|}\big)^{1/2}  \big(\E_x e^{2 (D(X_1)-\gamma|X(1)|^\beta)_+  }\big)^{1/2}.
\end{equation}
We take $\nu>0$ and $\rho \in (0,1)$ as in Lemma~\ref{six} and put $\lambda:=\nu /2$.
Then, thanks to Lemma~\ref{six}, for any $B>0$ there exists $N=N(B)>0$ such that for any $x\in\C_{B,N}$  we have
\begin{equation}\label{fterm}
(\E_x e^{2\lambda|X(1)|}\big)^{1/2}\le e^{\lambda|x(0)|}(1-\rho)^{1/2}\le V(x)(1-\rho/2).
\end{equation}
Thus, it remains to show that on $\C_{B,N}$ the second factor in the right--hand side of \eqref{secondfactor} is  smaller than
$(1+ \rho/2)$. Without loss of generality we assume that $N(B)$ is large enough so that $B N^\beta \le N$ (otherwise we can take larger $N(B)$). Using the inequality $|X(1)|\ge |x(0)|-D(X_1)$, we deduce for any $x\in\C_{B,N}$
\begin{align*}
\E_x e^{2(D(X_1)-\gamma|X(1)|^\beta)_+}\le& 1 +\E_x e^{2 (D(X_1)-\gamma|X(1)|^\beta)}\\
\le& 1+ e^{-2\gamma|x(0)|^\beta}\E_x e^{2 (D(X_1)+\gamma D(X_1)^\beta)}\\
\le& 1+  e^{-2(\gamma|x(0)|^\beta-C_\gamma)}\E_x e^{4 D(X_1)},
\end{align*}
where we have also used the fact that for some $C_\gamma>0$ we have $\gamma z^\beta\le z+C_\gamma$ for all $z\ge0$. We continue this estimate,
using Lemma~\ref{L:51}. Recall that on $\C_{B,N}$ we also have $D(x)\le |x(0)|$, thanks to our additional assumption on $N(B)$. Therefore we derive for any $x\in\C_{B,N}$
\begin{equation}\label{laststep}
\E_x e^{2 (D(X_1)-\gamma|X(1)|^\beta)_+}\le 1+  \exp\big\{-|x(0)|^\beta(2\gamma-
C) +C+2C_\gamma\big\}
\end{equation}
and the constant $C$ depends neither on $\gamma$ nor on $B$. Thus, taking $\gamma=\gamma(C)$ large enough, and combining \eqref{secondfactor},
\eqref{fterm} and \eqref{laststep}, we see that for any $B>0$ there exists a constant $N_1(B)$ such that on $\C_{B,N_1}$ we have
\begin{equation}\label{pochtidoshli}
\E_xV(X_1)\le V(x) (1-\rho/4),\quad x\in \C_{B,N_1}.
\end{equation}
Now with such $\lambda$ and $\gamma$ in hand we apply Lemma~\ref{three} with $\eps=1-\rho/4$. We get that there exist $B=B(\lambda,\gamma)>0$, $L=L(\lambda,\gamma)>0$ such that
\begin{equation*}
\E_xV(X_1)\le V(x) (1-\rho/4),\quad D(x)\ge (L\vee B|x(0)|^\beta).
\end{equation*}
Together with \eqref{pochtidoshli} this bound implies that for some $N_2=N_2(\lambda,\gamma)$, $L_1=L_1(\lambda,\gamma)$ we have
\begin{equation*}
\E_xV(X_1)\le V(x) (1-\rho/4),\quad |x(0)|\ge N_2\text{ or }D(x)\ge L_1.
\end{equation*}
Finally, if $|x(0)|\le N_2$ and $D(x)\le L_1$, then by \eqref{lemmasup}
\begin{equation*}
\sup_{\substack{|x(0)|\le N_2\\ D(x)\le L_1}}\E_x V(X_1)<\infty.
\end{equation*}
This completes the proof of the lemma.
\end{proof}

Based on the previous lemmas, we can now complete the proof of Theorem \ref{T:1}.

\begin{proof}[Proof of Theorem~\ref{T:1}]
It follows from Lemma~\ref{L:4.8} that condition \eqref{Lyap} holds with the function $\Psi(z):=c_1 z$, $z\in\R_+$, where the constant $c_1$ is defined in \eqref{statement4.8}. Hence  Theorem~\ref{T:1} follows immediately from Propositions~\ref{P:41} and \ref{P:42}.
\end{proof}

\subsection{Proof of Theorem~\ref{T:2}}\label{S:43}

Now we move on to the subgeometric case. We fix till the end of this section the constants $\alpha$, $\sigma$, $M$  and the function $\kappa$ from Assumption \textbf{\ref{A:3}}. As above without loss of generality, we assume that
$\kappa$ is an increasing concave function.

We work with a Lyapunov function
\begin{equation}\label{vau2}
V(x):=\exp\bigl(\lambda_1|x(0)|^\alpha+\lambda_2\bigl(D(x)^2-\psi(|x(0)|)\bigr)_+\bigr),
\end{equation}
where $\lambda_1>0$, $\lambda_2>0$, and $\psi\colon\R_+\to\R_+$ is an increasing continuous concave function. We will specify $\lambda_1$, $\lambda_2$, and the function $\psi$ later.
As before, we consider two cases: the ``small'' diameter case (where we gain from the decrease of the first factor in the Lyapunov function) and the ``large'' diameter case (where we gain from the decrease of the second factor in the Lyapunov function).

We start with the second case. Recall the definition of $\lambda_0$ from Lemma~\ref{L:51}.
\begin{Lemma}\label{L:53}
Suppose that Assumptions \textbf{\ref{A:1}} and \textbf{\ref{A:3}} hold. Let $V$ be the Lyapunov function defined in \eqref{vau2}. For any $\lambda_1>0$, $\lambda_2\in(0,\lambda_0/2]$ there exists $R=R(\lambda_1,\lambda_2)$ such that for any $x\in\C$
with $D(x)> (R+\psi(|x(0)|))^{1/2}$ we have
\begin{equation}\label{pervoeusl}
\E_x V(X_1)\le V(x)/2.
\end{equation}

Further, for any $\lambda_1>0$, $\lambda_2\in(0,\lambda_0/2]$, $N>0$ we have
\begin{equation}\label{obshayagran}
\sup_{x\in\C,\, |x(0)|\le N}\E_x V(X_1)<\infty
\end{equation}

\end{Lemma}
\begin{proof}
For any $x\in\C$ we derive
\begin{align}\label{nachalo49}
\E_x V(X_1)=&\E_x \exp\bigl(\lambda_1|X(1)|^\alpha+\lambda_2\bigl(D(X_1)^2-\psi(|X(1)|)\bigr)_+\bigr)\nn\\
\le& e^{\lambda_1|x(0)|^\alpha}\bigl(\E_x e^{2\lambda_1 (D(X_1)+1)}\bigr)^{1/2}
\bigl(\E_x e^{2\lambda_2 D(X_1)^2}\bigr)^{1/2}\nn\\
\le& C e^{C (\lambda_1+\lambda_1^2)} e^{\lambda_1|x(0)|^\alpha}
\end{align}
where in the second inequality we applied Lemma~\ref{L:51} and used the fact that $2\lambda_2\le \lambda_0$. This immediately implies \eqref{obshayagran}. To establish \eqref{pervoeusl} we deduce from \eqref{nachalo49} that
\begin{equation*}
\E_x V(X_1)\le V(x)e^{-\lambda_2(D(x)^2-\psi(|x(0)|))_+}C e^{C (\lambda_1+\lambda_1^2)},\quad x\in\C.
\end{equation*}
Now we find large $R=R(\lambda_1,\lambda_2)$ such that
\begin{equation*}
e^{-\lambda_2R}C e^{C (\lambda_1+\lambda_1^2)}<1/2.
\end{equation*}
By above, if $D(x)> (R+\psi(|x(0)|))^{1/2}$, then $\E_x V(X_1)\le V(x)/2$.
\end{proof}

Now we move on to the ``small'' diameter case. Recall the definition of the constant $C_\kappa$
from the beginning of Section~\ref{S:42}.

\begin{Lemma}\label{L:54} Suppose that the assumptions of Lemma~\ref{L:53} hold. Then there exist $\nu>0$, $N_0>0$, $\rho_1>0$, $\rho_2>0$  such that for any $x\in\C$ with $|x(0)|\ge N_0$,
$D(x)\le \kappa(|x(0)|)/(4C_\kappa)$  we have
\begin{equation}\label{second}
\E_x e^{\nu|X(1)|^\alpha}\le e^{\nu|x(0)|^\alpha}(1-\rho_1|x(0)|^{2\alpha-2})+\rho_2.
\end{equation}

\end{Lemma}
\begin{proof}
Recall the definition of $M$ from condition \eqref{cond}. Let $\lambda\in(0,1)$. Similar to the proof of Lemma~\ref{six},
we introduce a smooth function $\phi:\R^d\to\R_+$ such that for $|y|\ge M$ we have $\phi(y)=\exp(\lambda |y|^\alpha)$
and for $|y|\le M$ we have $\phi(y)\in[0,e^{\lambda M}]$. Arguing as in the proof of
Lemma~\ref{six}
and applying Ito's formula, we have for any $x\in\C$
\begin{align*}
\E_x \phi (X(1))\le& \phi(x)+\alpha\lambda\int_0^1 \E_x \I (|X(s)|\ge M)e^{\lambda |X(s)|^{\alpha}}|X(s)|^{\alpha-2}\langle X(s),f(X_s)\rangle \, ds\\
&+\frac12 \lambda\alpha\int_0^1 \E_x  \I (|X(s)|\ge M) e^{\lambda |X(s)|^\alpha}|X(s)|^{\alpha-2}(C_1 \lambda\alpha|X(s)|^{\alpha}+C_1)ds+ C_1\\
\le& \phi(x) + \alpha\lambda (I_1+ I_2)+C_1.
\end{align*}
Suppose further that  $D(x)\le \kappa(|x(0)|)/(4C_\kappa)$. First we bound $I_1$.
Using assumption \eqref{cond}, we derive
\begin{align*}
I_1\le& -\sigma\int_0^1 \E_x \I (|X(s)|\ge M)e^{\lambda |X(s)|^{\alpha}} |X(s)|^{2 \alpha-2}\,ds\\
&+C_2\int_0^1 \E_x\I(D(X_s)>\kappa(|X(s)|), |X(s)|\ge M) e^{\lambda |X(s)|^{\alpha}}|X(s)|^{\alpha-1}\,ds + C_3.
\end{align*}
It follows from the Cauchy--Schwarz inequality that for any $s\in[0,1]$
\begin{align*}
\E_x\I&(D(X_s)>\kappa(|X(s)|), |X(s)|\ge M) e^{\lambda |X(s)|^{\alpha}}|X(s)|^{\alpha-1}\\
&\le M^{\alpha-1}\Bigl(\P_x(D(X_s)>\kappa(|X(s)|))\Bigr)^{1/2} \Bigl(\E_x \exp(2\lambda|X(s)|^{\alpha})\Bigr)^{1/2}\\
&\le C_4 M^{\alpha-1} e^{-C_5 \kappa(|x(0)|)^2}e^{\lambda |x(0)|^\alpha}e^{C_6(1+\lambda+\lambda^2)},
\end{align*}
where in the last inequality we use the bound $|X(s)|^\alpha\le |x(0)|^{\alpha} + D(X_1)+1$, and estimates \eqref{radius} and \eqref{psiest}.

To bound $I_2$ we take a large $\widetilde M>M$ to get
\begin{equation*}
I_2\le \int_0^1 \E_x  \I (X(s)\ge \widetilde M) e^{\lambda |X(s)|^\alpha}|X(s)|^{2\alpha-2}(C_1 \lambda\alpha+C_1\widetilde M^{-\alpha})ds+C_7e^{\widetilde M}.
\end{equation*}

Note that the constants $C_1$, $C_2$, ... $C_7$ above do not depend on $\lambda\in[0,1]$ or $\widetilde M$.

Combining all the previous estimates, we deduce
\begin{align}\label{predlast}
\E_x \phi (X(1))\le& \phi(x) -\theta \int_0^1 \E_x \I (|X(s)|\ge M)\phi(X(s)) |X(s)|^{2 \alpha-2}\,ds +C_8e^{\widetilde M}\nn\\
&+C_9e^{-C_5 \kappa(|x(0)|)^2}e^{\lambda |x(0)|^\alpha}.
\end{align}
where $\theta:=\alpha\lambda(\sigma- C_1\lambda \alpha-C_1{\widetilde M}^{-\alpha})$.

Recall that $C_1$ does not depend on $\widetilde M$ and $\lambda$. Thus, we take $\lambda$ to be small enough and $\widetilde M$ to be large enough so that
$\theta>0$.

Suppose additionally that $|x(0)|\ge 1$. We derive for  $s\in[0,1]$
\begin{align}\label{lastexspr}
\E_x \bigl[\I (|X(s)|&\ge M)\phi(X(s)) |X(s)|^{2 \alpha-2}\bigr]\nn\\
\ge&\E_x \bigl[\I (|X(s)|\ge M)\phi(X(s)) (1+|X(s)|)^{2 \alpha-2}\bigr]\nn\\
\ge& \E_x \bigl[e^{\lambda|X(s)|^\alpha} (1+|X(s)|)^{2 \alpha-2}\bigr]-C_{10}\nn\\
\ge&  e^{\lambda|x(0)|^\alpha}|x(0)|^{2\alpha-2}
\E_x \bigl[e^{-D(X_1)- 1} (1+|x(0)|^{-1}+D(X_1)|x(0)|^{-1})^{2 \alpha-2}\bigr]-C_{10}\nn\\
\ge& e^{\lambda|x(0)|^\alpha}|x(0)|^{2\alpha-2}
\E_x \bigl[e^{-D(X_1)- 1} (2+D(X_1))^{2 \alpha-2}\bigr]-C_{10}.
\end{align}
We continue the calculations using Jensen's inequality and estimate \eqref{radius} with $\beta=0$. We get
\begin{equation*}
\E_x \bigl[e^{-D(X_1)- 1} (2+D(X_1))^{2 \alpha-2}\bigr]\ge
\Bigl(\E_x \bigl[e^{D(X_1)+ 1} (2+D(X_1))^{2 -2\alpha}\bigr]\Bigr)^{-1}\ge C_{11}>0.
\end{equation*}
Combining this estimate  with \eqref{lastexspr} and \eqref{predlast}, we get
\begin{equation}\label{nownow}
\E_x e^{\lambda|X(1)|^\alpha}\le e^{\lambda|x(0)|^\alpha} -C_{12} e^{\lambda|x(0)|^\alpha}|x(0)|^{2\alpha-2}+ C_9e^{\lambda |x(0)|^\alpha}e^{-C_5 \kappa(|x(0)|)^2}+C_{13}.
\end{equation}
Recall our assumption on growth of $\kappa$: $\kappa^2(z)/\log(z)\to\infty$ as $z\to\infty$. Thus,
if $N_0$ is large enough, and $|x(0)|>N_0$, then  inequality \eqref{second} with $\nu=\lambda$
follows directly from \eqref{nownow}.
\end{proof}

Now we are ready to give the proof of Theorem~\ref{T:2}. Recall the definitions of $N_0$, $\nu$ from Lemma~\ref{L:54}. The next lemma is crucial.

\begin{Lemma}\label{L:55}
Suppose that the assumptions of Lemma~\ref{L:53} hold. Then there exist $\lambda_1>0$, $\lambda_2>0$, $c_1>0$, $c_2>0$, $N>0$
and a function $\psi$ such that
\begin{equation}\label{estv2}
\E_x V(X_1)\le V(x)(1-c_1|x(0)|^{2\alpha-2})+c_2, \quad x\in\C, |x(0)|\ge N.
\end{equation}
\end{Lemma}
\begin{proof}
First we fix the function $\psi$ such that $\psi(t)/\log(t)\to\infty$ and $\kappa(t)^2/\psi(t)\to\infty$ as $t\to\infty$. Such $\psi$ exists due to our assumptions on the growth of $\kappa$. Take $C_\psi$ such that
$\psi(t)\le t + C_\psi$, $t\ge0$.

We begin  in the same way as in Lemma~\ref{L:53}. We derive for any $x\in\C$
\begin{align}\label{fstep2}
\E_x V(X_1)=&\E_x \exp\Bigl(\lambda_1|X(1)|^\alpha+\lambda_2\bigl(D(X_1)^2-\psi(|X(1)|)\bigr)_+\Bigr)\nn\\
\le& \bigl(\E_x e^{2\lambda_1 |X(1)|^\alpha}\bigr)^{1/2}
\bigl(\E_x e^{2\lambda_2 (D(X_1)^2-\psi(|X(1)|))_+}\bigr)^{1/2}.
\end{align}

First we deal with the second factor in the right-hand side of \eqref{fstep2}. We have
\begin{align}\label{fstep21pred}
\E_x e^{2\lambda_2 (D(X_1)^2-\psi(|X(1)|))_+}\le& 1+ \E_x e^{2\lambda_2 (D(X_1)^2-\psi(|X(1)|))}\nn\\
\le&1+ e^{-2\lambda_2\psi(|x(0)|)}\E_x e^{2\lambda_2 [D(X_1)^2+\psi(D(X_1))]}\nn\\
\le&1+ e^{-2\lambda_2(\psi(|x(0)|)-C_\psi-1)}\E_x e^{4\lambda_2 D(X_1)^2}.
\end{align}
We choose $\lambda_2:=\lambda_0/4$. Then \eqref{fstep21pred} and \eqref{firsteq} imply
\begin{equation}\label{fstep21}
\E_x e^{2\lambda_2 (D(X_1)^2-\psi(|X(1)|))_+}\le1+ C e^{-2\lambda_2\psi(|x(0)|)},\quad x\in\C.
\end{equation}
The first factor in the right-hand side of \eqref{fstep2} has been already estimated in Lemma~\ref{L:54}. We put
$\lambda_1:=\nu/2$.
If $|x(0)|\ge N_0$ and $D(x)\le \kappa(|x(0)|)/(4C_\kappa)$, then
\eqref{second}, \eqref{fstep2}, \eqref{fstep21} imply for such $x$
\begin{equation*}
\E_x V(X_1)\le
 e^{\lambda_1|x(0)|^\alpha}(1-C_1|x(0)|^{2\alpha-2}+C_2e^{-2\lambda_2\psi(|x(0)|)})+C_3.
\end{equation*}
Since $\psi(z)/\log(z)\to\infty$ as $z\to\infty$, we get for large enough $N_1$ and all $x\in\C$ with $|x(0)|\ge N_1$ and
$D(x)\le \kappa(|x(0)|)/(4C_\kappa)$
\begin{equation}\label{sluchai1}
\makebox[0pt]{$\E_x V(X_1)\le
 e^{\lambda_1|x(0)|^\alpha}(1-C_4|x(0)|^{2\alpha-2})+C_5\le V(x)(1-C_4|x(0)|^{2\alpha-2})+C_5.$}
\end{equation}
Now let us consider the second case: $D(x)\ge \kappa(|x(0)|)/(4C_\kappa)$. Choose $R$ as in Lemma~\ref{L:53}.
Due to our assumptions on the growth of $\psi$, for large enough $N_2$ and any $z>N_2$ we have $\kappa(z)/(4C_\kappa)\ge (R+\psi(z))^{1/2}$. Thus, by
Lemma~\ref{L:53} we have for all $x\in\C$ with $|x(0)|\ge N_2$ and
$D(x)\ge \kappa(|x(0)|)/(4C_\kappa)$
\begin{equation*}
\E_x V(X_1)\le V(x)/2.
\end{equation*}
This together with \eqref{sluchai1}  proves the lemma.
\end{proof}

\begin{proof}[Proof of Theorem~\ref{T:2}]
 Let $V$ be a Lyapunov function defined in \eqref{vau2} with the parameters specified in Lemma~\ref{L:55}. Let $\Psi\colon\R_+\to\R_+$ be a differentiable concave increasing function such that
\begin{equation*}
\Psi(z)=\frac{z}{(\log z)^{(2-2\alpha)/\alpha}}
\end{equation*}
for large enough $z$ (more precisely, for $z\ge N$ where $N>0$ is the same as in Lemma~\ref{L:55}).
 It follows from \eqref{estv2} that for $x\in\C$ with $|x(0)|\ge N$ 
\begin{equation*}
\E_x V(X_1)\le V(x)-C_1 V(x) (\log V(x))^{(2\alpha-2)/\alpha}+C_2.
\end{equation*}
This together with \eqref{obshayagran}  implies that for any $x\in\C$ 
\begin{equation*}
\E_x V(X_1)\le V(x)-C_1 \Psi(V(x))+C_3.
\end{equation*}
Thus, condition \eqref{Lyap} holds with the function $\Psi$ defined above. Therefore Theorem~\ref{T:2} follows now from Propositions~\ref{P:41} and \ref{P:42}.
\end{proof}

\end{document}